\documentclass[12pt]{article}

\usepackage{graphics,amsthm, amsmath, amssymb}

\usepackage{color,caption}

\usepackage{graphics}

\usepackage{graphicx}
\usepackage[normalem]{ulem}

\expandafter\chardef\csname pre amssym.def at\endcsname=\the\catcode`\@
\catcode`\@=11

\def\undefine#1{\let#1\undefined}
\def\newsymbol#1#2#3#4#5{\let\next@\relax
 \ifnum#2=\@ne\let\next@\msafam@\else
 \ifnum#2=\tw@\let\next@\msbfam@\fi\fi
 \mathchardef#1="#3\next@#4#5}
\def\mathhexbox@#1#2#3{\relax
 \ifmmode\mathpalette{}{\m@th\mathchar"#1#2#3}%
 \else\leavevmode\hbox{$\m@th\mathchar"#1#2#3$}\fi}
\def\hexnumber@#1{\ifcase#1 0\or 1\or 2\or 3\or 4\or 5\or 6\or 7\or 8\or
 9\or A\or B\or C\or D\or E\or F\fi}

\font\tenmsa=msam10
\font\sevenmsa=msam7
\font\fivemsa=msam5
\newfam\msafam
\textfont\msafam=\tenmsa
\scriptfont\msafam=\sevenmsa
\scriptscriptfont\msafam=\fivemsa
\edef\msafam@{\hexnumber@\msafam}
\mathchardef\dabar@"0\msafam@39
\def\dashrightarrow{\mathrel{\dabar@\dabar@\mathchar"0\msafam@4B}}
\def\dashleftarrow{\mathrel{\mathchar"0\msafam@4C\dabar@\dabar@}}

\font\tenmsb=msbm10
\font\sevenmsb=msbm7
\font\fivemsb=msbm5
\newfam\msbfam
\textfont\msbfam=\tenmsb
\scriptfont\msbfam=\sevenmsb
\scriptscriptfont\msbfam=\fivemsb
\edef\msbfam@{\hexnumber@\msbfam}

\theoremstyle{plain}
\newtheorem{theorem}{Theorem}[section]
\newtheorem{lemma}[theorem]{Lemma}

\newtheorem{corollary}[theorem]{Corollary}

\theoremstyle{definition}

\newtheorem{example}[theorem]{Example}

\def\deg{{\rm deg}}

\def\cP{{\mathcal P}}

\def\cc{{\mathcal C}}

\def\qed{\hfill  \framebox(5,5){}}

\def\paraalg{\vspace{2 mm}}

\def\R{{\mathbb R}}
\def\C{{\mathbb C}}

\def\lim{{\rm lim}}

\def\deg{{\rm deg}}

\def\cP{{\mathcal P}}

\def\cc{{\mathcal C}}

\def\ii{{\mathrm i}}

\def\P{{\mathbf P}}
\def\cP{{\mathbf P}}
\def\c{{\mathbf C}}
\def\cc{{\mathbf C}}
\def\S{{\mathbf S}}
\def\crit{{\rm Crit}(\P)}
\def\qed{\hfill  \framebox(5,5){}}
\def\para{\vspace{1 mm}}
\def\C{{\mathbb C}}
\def\ii{{\emph \,i}\,}
\def\r{{\mathbf r}}
\def\cir{{\mathbf c}{\mathbf i}{\mathbf r}{\mathbf c}}

\begin{document}

\title{Missing sets in rational parametrizations of surfaces of revolution
\thanks{rafael.sendra@uah.es, sevillad@unex.es, carlos.villarino@uah.es} }

\author{
 J. Rafael  Sendra, Carlos Villarino \\
Dpto. de Fisica y Matem\'aticas,
        Universidad de Alcal\'a \\
      E-28871 Madrid, Spain  \\
\and
David Sevilla \\
Centro Universitario de M\'erida \\ Av. Santa Teresa de Jornet 38 \\
E-06800, M\'erida, Badajoz, Spain
}
\date{}          
\maketitle

\begin{abstract}
Parametric representations do not cover, in general, the whole geometric object that they parametrize. This can be a problem in practical applications. In this paper we analyze the question for surfaces of revolution generated by real rational profile curves, and we describe a simple small superset of the real zone of the surface not covered by the parametrization. This superset consists, in the worst case, of the union of a circle and the mirror curve of the profile curve.
\end{abstract}

\section{Introduction}

Parametric representations of structured surfaces like ruled surfaces, surfaces of revolution or swept surfaces are often used in computer graphics, CAD/CAM, and surface/geometric modelling (see e.g. \cite{Agoston}, \cite{marsh}).  Nevertheless, when working with parametric instead of implicit representations, one must take into account that some information of the geometric object can be missed. More precisely, the parametrization may not cover the whole object, that is, some part of the object may not be reachable by giving values to the parameters; for instance, the curve parametrization $(\frac{2t}{t^2+1},\frac{t^2-1}{t^2+1})$ covers the unit circle with the exception of the point $(0,1)$. A curve parametrization may miss, at most, one point, called the critical point (see \cite{andradas} or  \cite{Sendra2002a}). However, a surface parametrization may miss finitely many curves and finitely many points; this is a consequence of the fact that the image of the  parametrization is a constructible set of the surface.  We will refer to the uncovered part as the \emph{missing set} of the parametrization.

We observe that the phenomenon described above can be seen as a particular case of the geometric covering problem (see e.g. \cite{Shragai}), in the sense that the image of the  parametrization is the subset that one guard covers, and the missing set is the inspection location to be covered by other guards.

Parametrizations with nonempty missing sets can be a problem in practical applications if there is relevant information outside the covered part. Examples of this claim can be found in \cite{SendraSevillaVillarino2014b} (for the computation of intersections), in \cite{SendraSevillaVillarino2014C} (for estimating Hausdorff distances) or \cite{SendraSevillaVillarino2014a} (for the analysis of cross sections).
One way to deal with this difficulty is to find parametrizations that do cover the whole object. In the curve case, there are algorithmic methods for that (see \cite{Sendra2002a}). However, the situation for surfaces is much more complicated, and, at least to our knowledge, it is an open problem. Instead, one may use other alternatives. For instance, in \cite{bajaj}, \cite{SendraSevillaVillarino2014a}, \cite{SendraSevillaVillarino2014C},
the authors compute finitely many parametrizations such that their images cover all the surface. Another possibility is to have a precise description of the missing set of the parametrization, or a subset of the surface containing the missing set; a subset of the surface, containing the missing set and having dimension smaller than 2, is called a \emph{critical set}. In this way, for a practical application one can use the parametrization and then decide the existence of relevant points in the critical set.

The last strategy can be approached by using elimination theory techniques (see \cite{SendraSevillaVillarino2014b}). Nevertheless, although theoretically possible, the direct use of these techniques produces, in general, huge critical sets and requires solving systems of algebraic equations. As a consequence, the method turns to be inefficient in practice. However, when working with structured surfaces, a preliminary analysis of the structure can help to describe quickly and easily a critical set. For instance, in \cite{SendraSevillaVillarino2014C}, we show that any rational ruled surface can be parametrized so that the critical set is a line which is easily computable from the parametrization. In this paper, we analyze the case of surfaces of revolution given by means of a real plane curve parametrization known as a profile. We prove that a critical set for the real part of a  surface of revolution is, in the worst case, the union of a curve (the mirror curve of the profile curve) and a circle passing through the critical point of the profile curve; see Table \ref{tabla}. As a direct criterion (see Corollary \ref{cor}), we obtain that any parametrization of a symmetric real curve with at least one polynomial component generates all the real part of the surface of revolution.

As we will see in the subsequent sections this critical set is indeed very simple to compute from the profile curve parametrization. An additional advantage of our method is that it does not require that the parametrization of the surface is proper (i.e. injective), while the direct application of elimination techniques needs to compute the inverse of the parametrization, and hence requires that the surface parametrization is proper.

The rest of the paper is organized as follows. In Section \ref{sec-results} we present the main results of the paper. The proofs of these results appear in the appendix. In Section \ref{sec-algorithm} we outline the algorithmic methods derived from the theoretical results, and we illustrate them by some examples. Future work on the topic is discussed in Section \ref{sec-extensions}. The paper ends with a brief conclusion.

Computations were performed with the mathematical software Maple 18. Plots were generated with Maple and Surfer.

\section{Results}\label{sec-results}

Let $\c^{\rm P}$ be a curve (\emph{profile curve})  in the $(y,z)$-plane parame\-tri\-zed by $\r^{\rm P}(t)=(0,p(t),q(t))$, where $p(t), q(t)$ are rational functions with real coefficients; the results presented here are also valid if the coefficients are complex numbers but for simplicity, and because of the interest in applications, we limit the setting to the real case. In addition, we assume that $\r^{\rm P}$ is proper, that is, injective. We observe that every non-proper parametrization can be reparametrized into a proper one (see for example Section 6.1. in \cite{SWP}).
Also let $\S$ be the surface of revolution generated by rotating $\cc^{\rm P}$ around the $z$-axis. We exclude the trivial case where $\cc^{\rm P}$ is a line parallel to the $y$ axis, in which $\S$ is a plane. The classical parametrization of $\S$, obtained from $\r^{\rm P}(t)$, is
\[
 \P(s,t)=\left(\frac{2s}{1+s^2} \, p(t),\frac{1-s^2}{1+s^2} \, p(t),q(t)\right).
\]
Observe that properness is assumed in the profile parametrization $\r^{\rm P}(t)$ but not in $\P(s,t)$; see Example 2.3 in \cite{ARSTV} for an example where $\P(s,t)$ is non-proper and $\r^{\rm P}(t)$ is proper.

In addition, we consider the parametric curve $\c^{{\rm M}}$ (called \emph{mirror curve of $\c^{\rm P}$})  parametrized as
$\r^{\rm M}(t)=(0,-p(t),q(t))$. Observe that $\c^{\rm P}=\c^{\rm M}$ if and only if $\c^{\rm P}$ symmetric with respect to the $z$-axis. For instance, the parabola $y=z^2$ is equal to its mirror curve while the cubic $y=z^3$ is not. Finally, we represent by
$\cir(\alpha,c)$
 the circle of radius $|\alpha|$ in the plane $z=c$ centered at $(0,0,c)$, that is, the curve parametrized as
\[
 \left(\frac{2s}{1+s^2} \, \alpha,\frac{1-s^2}{1+s^2} \, \alpha, c\right).
\]
Observe that $\cP(s,t_0)$ is $\cir(p(t_0),q(t_0))$, i.e. the {\em  cross section circle} of the surface $\S$ of revolution passing through $(0,p(t_0),q(t_0))$.

Before, stating our main results, we need to recall the notion of normal (i.e. surjective) curve parametrization and critical point, for further details see \cite{Sendra2002a} and \cite{andradas}. We say that a curve parametrization $\r(t)$ is normal if all points on the curve are reachable by $\r(t)$ when $t$ takes values in the field of the complex numbers. The theory establishes that a proper curve parametrization can miss at most one point. This point is called the {\em critical point}, and it can be seen, in the case of real functions, as the limit, when $t$ goes to $\infty$, of the parametrization; understanding that if this limit does not exist then there is no critical point and the parametrization is normal. For instance, $(t,t^2)$ or $(t,1/t)$ are normal, but the circle parametrization $$\left(\frac{2t}{t^2+1},\frac{t^2-1}{t^2+1}\right)$$ has the {\it East pole} $(0,1)$ as  critical point; indeed, it is not reachable and hence the parametrization is not normal.
In this situation, we are ready to established our main results (see the Appendix for the formal proof of the results).

\subsection{The real case}\label{subsec-real-case}

We describe now a critical set of the real part of $\S$ that $\cP(s,t)$ does not cover. For this purpose, we distinguish whether the profile curve is symmetric or not. The next theorem states that, in the symmetric case, at most one point can be missed in the real part of the surface of revolution.

\para

\begin{theorem}\label{theorem-sym-real} {\sf [Symmetric real case]} Let $\c^{\rm P}$ be symmetric.
\begin{enumerate}
\item If $\r^{\rm P}(t)$ is normal, the empty set is a real-critical set of $\cP(s,t)$.
\item If $\r^{\rm P}(t)$ is not normal, and $(b,c)$ is its critical point, then $\{(0,b,c)\}$ is a real-critical set of $\cP(s,t)$.
\end{enumerate}
\end{theorem}

Based on the previous theorem and on Theorem 2.8 in \cite{Gao} one has the following corollary.

\begin{corollary}\label{cor}
If $\r^{\rm P}$ is symmetric, and at least one of its components has a numerator of degree greater than the degree of the denominator, then $\P(s,t)$ covers all $\S$.
\end{corollary}

The next theorem states that, in the non-symmetric case, the missing  real part of the surface of revolution is included in the union of the mirror curve and either the critical point of the profile curve or a cross-section circle.

\para

\begin{theorem}\label{theorem-non-sym-real} {\sf [Non-Symmetric real case]} Let $\c^{\rm P}$ be non-symmetric.
\begin{enumerate}
\item If $\r^{\rm P}(t)$ is normal, $\c^{\rm M}$ is a real-critical set of $\cP(s,t)$.
\item If $\r^{\rm P}(t)$ is not normal, and $(0,b,c)$ its critical point, then a real-critical set of $\cP(s,t)$ is $\c^{\rm M}$ if $(0,-b,c)\in \c^{\rm P}$, otherwise $\c^{\rm M}\cup \cir(b,c)$ is.
\end{enumerate}
\end{theorem}

 \subsection{The complex case}\label{subsec-complex-case}

 Next, we describe a critical set when the revolution surface in embedded in the complex space. The next theorem states that, in the (complex)  case, besides the real missing part introduced in Theorems \ref{theorem-sym-real} and \ref{theorem-non-sym-real}, one may miss pairs of complex lines settled at each (real or complex) intersection of $\c^{\rm P}$ with the $z$-axis.

\para

\begin{theorem}\label{theorem-complex} {\sf [Complex case]} Let ${\cal A}$ be the real critical set of $\cP(s,t)$ provided by Theorems \ref{theorem-sym-real} and \ref{theorem-non-sym-real}, and let $J$ be the set of all (real and complex) $z$-coordinates of the intersection points of $\c^{\rm P}$ with the $z$-axis. Then, a complex-critical set of $\cP(s,t)$ is
\[ {\cal A}\,\bigcup_{\lambda \in J} \{(t,\pm \ii t,\lambda) \,|\, t\in \C \}. \]
\end{theorem}

\section{Algorithmic framework}\label{sec-algorithm}

The theorems presented in Section \ref{sec-results} provide algorithmic processes to cover the surface of revolution $\S$. In the following we outline them.

\paraalg

\hrule

\para

\noindent {\bf Algorithm 1:} Covering the real part of a surface of revolution
\hrule

\para

\noindent {\bf Input:} a real proper parametrization $\r^{\rm P}(t)=(0,p(t),q(t))$ of the profile curve $\c^{\rm P}$.

\para

\noindent {\bf Output:} a real-critical set for $\r^{\rm P}$.

\begin{tabbing}
aaa \=  a \= a \= a \= a \= a \kill
\> {\sf 1.1.} \= {\bf If} \= $\c^{\rm P}=\c^{\rm M}$  \= {\bf then} \=  {\bf if} \= $\r^{\rm P}(t)$ is normal   {\bf return}  $\emptyset$ \\
\> {\sf 1.2.} \>     \> \> \> \> {\bf else} \= let $A$ be the critical point of $\r^{\rm P}(t)$   \\
\> {\sf 1.3.} \>       \> \> \> \> \> {\bf return} $\{A\}$  \\
\> {\sf 1.4.} \> \>\> \> {\bf end if}\\
\> {\sf 1.5.} \>       \>  {\bf else} \=  {\bf if}  \= $\r^{\rm P}(t)$ is normal \=  {\bf return} $\r^{\rm M}(t)$ \\
\> {\sf 1.6.} \>       \> \> \> {\bf else} \= let  $A=(0,b,c)$ be the critical point of  $\r^{\rm P}(t)$  \\
\> {\sf 1.7.} \>       \> \> \> \> {\bf if} $b=0$ \= {\bf  return} $\r^{\rm M}(t)$ \\
\> {\sf 1.8.} \>  \> \> \> \> {\bf else} \= {\bf if} \= $(0,-b,c)\in \c^{\rm P}$ \= {\bf  return} $\r^{\rm M}(t)$ \\
\> {\sf 1.9.} \>  \> \> \> \> \> \>  {\bf else return} $\r^{\rm M}(t)\cup \left(\frac{2s}{1+s^2}b,\frac{1-s^2}{1+s^2} b,c\right)$ \\
\> {\sf 1.10.} \>  \> \> \> \> \> {\bf end if} \\
\> {\sf 1.11.} \>  \> \> \> \> {\bf end if} \\
\> {\sf 1.12.}  \= \>   \>  {\bf end if} \\
\> {\sf 1.13.} \>  {\bf end if}
\end{tabbing}

\para

Let us comment some computational aspects of Algorithm 1. In  {\sf 1.1.} one needs to check whether $\c^{\rm P}=\c^{\rm M}$; for this we have several possibilities. One may apply the results in \cite{juange}, where the problem is analyzed without knowing the axis of symmetry. In our case, since this is known, namely $\{x=y=0\}$,  one may reason as follows. Let $m$ be the maximum of the degrees of $p(t)$ and $q(t)$. Since $\r^{\rm P}$ is proper, by Theorem 4.21 in \cite{SWP}, the degree of $\c^{\rm P}$ is bounded by $2m$; similarly, the degree of $\c^{\rm M}$ is also bounded by $2m$. Now, taking into account B\'ezout's theorem (see e.g. Theorem 2.48 in \cite{SWP}),  $\c^{\rm P}=\c^{\rm M}$ if and only if both curves share at least $2m+1$ common points. So,  we can proceed as follows.

\paraalg

\hrule

\para

\noindent {\bf Sub-Algorithm 1:} Checking symmetry
\hrule

\para

\noindent {\bf Input:} a real proper parametrization $\r^{\rm P}(t)=(0,p_1(t)/p_2(t),q_1(t)/q_2(t))$, with $\gcd(p_1,p_2)=\gcd(q_1,q_2)=1$, of the profile curve $\c^{\rm P}$.

\para

\noindent {\bf Output:} decision on the symmetry of $\c^{\rm P}$.

\begin{tabbing}
 aaa \= 1.11. \= \kill
\> {\sf 2.1.} \> $i:=1$, $t_0:=1$, $m:=\max\{\deg(p_1/p_2),\deg(q_1/q_2)\}$ \\
\> {\sf 2.2.} \> {\bf while} $i<2m+1$ {\bf do} \=
{\bf if} \= [$p_2(t_0)=0$ or $q_2(t_0)=0$] \\
\> {\sf 2.3.} \> \> \> $t_0:=t_0+1$ \\
\> {\sf 2.4.} \> \>  \> {\bf else} \\
\> {\sf 2.5.} \> \> \> $(\alpha,\beta):=\r^{\rm P}(t_0)$ \\
\> {\sf 2.6.} \> \> \> $h(t):=\gcd(\alpha p_2(t)+p_1(t),\beta q_2(t)-q_1(t))$ \\
\> {\sf 2.7.} \> \> \> {\bf if}    \= $\deg(h)>0$  then $i:=i+1$,  $t_0=t_0+1$. \\
\> {\sf 2.8.} \> \> \> \>     {\bf else return} $\c^{\rm P}\neq \c^{\rm M}$.\\
\> {\sf 2.9.} \> \> \>     {\bf end if} \\
\> {\sf 2.10.} \> \>    {\bf end if} \\
\> {\sf 2.11.}  \>   {\bf end do} \\
\> {\sf 2.12.} \> {\bf return} $\c^{\rm P}=\c^{\rm M}$.
\end{tabbing}

\para

Also in Step {\sf 1.1.} of Algorithm 1, one needs to check whether $\r^{\rm P}$ is normal. Recalling that $\r^{\rm P}$ is real and proper, one may proceed as follows (see Theorem 6.22. in \cite{SWP}).

\paraalg

\hrule

\para

\noindent {\bf Sub-Algorithm 2:} Checking normality
\hrule

\para

\noindent {\bf Input:} a real proper parametrization $\r^{\rm P}(t)=(0,p_1(t)/p_2(t),q_1(t)/q_2(t))$, with $\gcd(p_1,p_2)=\gcd(q_1,q_2)=1$, of the profile curve $\c^{\rm P}$.

\para

\noindent {\bf Output:} decision on the normality of $\c^{\rm P}$.

\para

\begin{tabbing}
 aaa \= 1.11. \= \kill
\> {\sf 3.1.} Compute $A:=\lim_{t \rightarrow \infty } \r^{\rm P}(t)$ \\
\> {\sf 3.2.} \= {\bf if} \= $A=(\alpha,\beta)\in \R^2$ \\
\> {\sf 3.3.} \> \>   $h(t):=\gcd(\alpha p_2(t)-p_1(t),\beta q_2(t)-q_1(t))$ \\
\>   {\sf 3.4.} \> \> {\bf if} \= $\deg(h)>0$ {\bf return} $\r^{\rm P}$ is normal \\
\> {\sf 3.5.} \> \> \> {\bf else} {\bf return} $\r^{\rm P}$ is not normal and $A$ is the critical point \\
\> {\sf 3.6.} \> \> {\bf end if} \\
\> {\sf 3.7.} \> \>  {\bf else} {\bf return} $\r^{\rm P}$ is  normal \\
\> {\sf 3.8.} \> {\bf end if}
 \end{tabbing}

\para

Finally, in Step {\sf 1.8.} one needs to check whether $(0,-b,c)\in \c^{\rm P}$, where $(0,b,c)$ is the critical point of $\r^{\rm P}$. With the information we have we can proceed as follows.

\paraalg

\hrule

\para

\noindent {\bf Sub-Algorithm 3:} Checking point
\hrule

\para

\noindent {\bf Input:} a real proper parametrization $\r^{\rm P}(t)=(0,p_1(t)/p_2(t),q_1(t)/q_2(t))$, with $\gcd(p_1,p_2)=\gcd(q_1,q_2)=1$, of the profile curve $\c^{\rm P}$; and the critical point $(0,b,c)$.

\para

\noindent {\bf Output:} decision on whether  $(0,-b,c)\in\c^{\rm P}$.

\begin{tabbing}
 aaa \= 1.11. \= \kill
\> {\sf 4.1.} \>   $h(t):=\gcd(b p_2(t)+p_1(t),c q_2(t)-q_1(t))$ \\
\> {\sf 4.2.} \>  {\bf if} \= $\deg(h)>0$ {\bf return} $(0,-b,c)\in \c^{\rm P}$. \\
\> {\sf 4.3.} \>  \> {\bf else} {\bf return} $(0,-b,c)\not\in \c^{\rm P}$. \\
\> {\sf 4.4.} \>  {\bf end if}
 \end{tabbing}

\paraalg

Similarly, we can derive the algorithm for covering the complex part of the surface.

\vspace*{2 mm}

\hrule
\para

\noindent {\bf Algorithm 2:} Covering the complex part of a surface of revolution
\hrule

\para

\noindent {\bf Input:} a real proper parametrization $\r^{\rm P}(t)=(0,p_1(t)/p_2(t),q_1(t)/q_2(t))$, with $\gcd(p_1,p_2)=\gcd(q_1,q_2)=1$, of the profile curve $\c^{\rm P}$.

\para

\noindent {\bf Output:} a set of complex parametrizations covering the complex part of the surface of revolution $\S$ generated  by $\r^{\rm P}$.

\begin{tabbing}
  a \= a \= a \= a \= a \kill
{\sf 2.1.} \=  {Apply} {\bf Algorithm 1}. Let $\cal A$ be the output. \\
{\sf 2.2.} \> {\bf return} ${\cal A} \cup \{(t,\ii t, q_1(\alpha)/q_2(\alpha))\,|\, p_2(\alpha)=0\}$.
\end{tabbing}

\para

We illustrate the algorithmic processes  by means of some examples.

\begin{example}\label{ex-11}
Let $\S$ be the surface of revolution generated by the profile curve
\[ \r^{\rm P}(t)=\left(0,\frac{t^5}{t^4+1}, \frac{t^2}{t^4+1}\right). \]
\begin{figure}[h]
\begin{center}
\includegraphics[width=6.2cm,height=6.3cm]{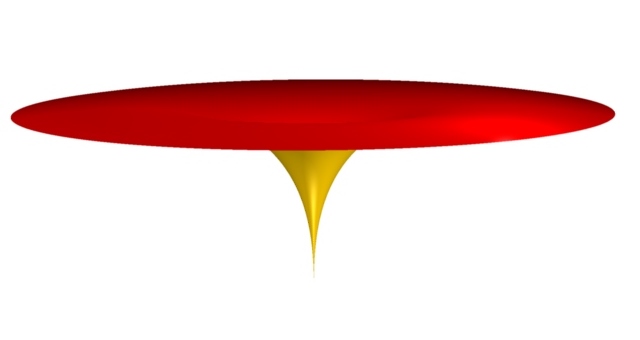}\hfill\includegraphics[width=4.5cm,height=4.5cm]{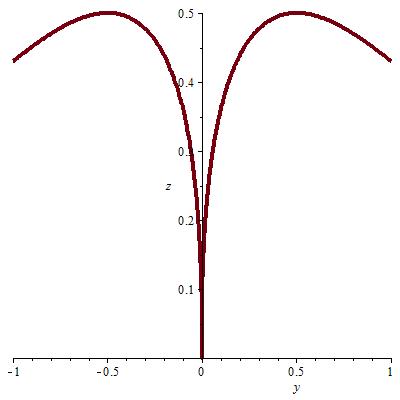}
\end{center}
\caption{Left: surface of revolution $\S$. Right: profile curve $\c^{\rm P}$ (Example \ref{ex-11})}
\end{figure}
$\c^{\rm P}=\c^{\rm M}$, and $\r^{\rm P}$ is normal. Thus all the real part of $\S$ is covered by $\P(s,t)$.
Algorithm 1 returns the empty set.
\end{example}

\begin{example}\label{ex-13}
Let $\S$ be the surface of revolution generated by the profile curve
\[ \r^{\rm P}(t)=\left(0,{\frac {t}{{t}^{4}+1}},{\frac {{t}^{2}-1}{{t}^{4}+1}}\right). \]
\begin{figure}[h]
\begin{center}
\includegraphics[width=7cm]{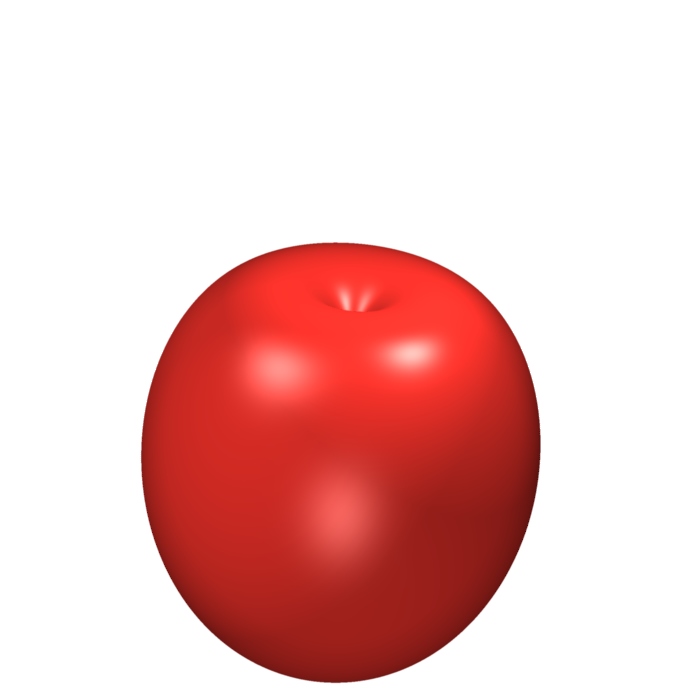}\hfill\includegraphics[width=4.2cm,height=4.2cm]{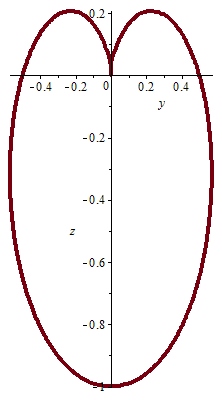}
\end{center}
\caption{Left: surface of revolution $\S$. Right: profile curve $\c^{\rm P}$ (Example \ref{ex-13})}
\end{figure}
$\c^{\rm P}=\c^{\rm M}$, and $\r^{\rm P}$ is not normal with $(0,0,0)$ as its critical point. Thus all the real part of $\S$, with the exception of the origin, is covered by $\P(s,t)$. Algorithm 1 returns $\{(0,0,0)\}$.
\end{example}

\para

\begin{example}\label{ex-15}
Let $\S$ be the surface of revolution generated by the profile curve
\[ \r^{\rm P}(t)=\left(0,{\frac {t}{{t}^{4}+1}},{\frac {{t}^{3}}{{t}^{2}+1}}\right). \]
\begin{figure}[h]
\begin{center}\includegraphics[width=5.5cm]{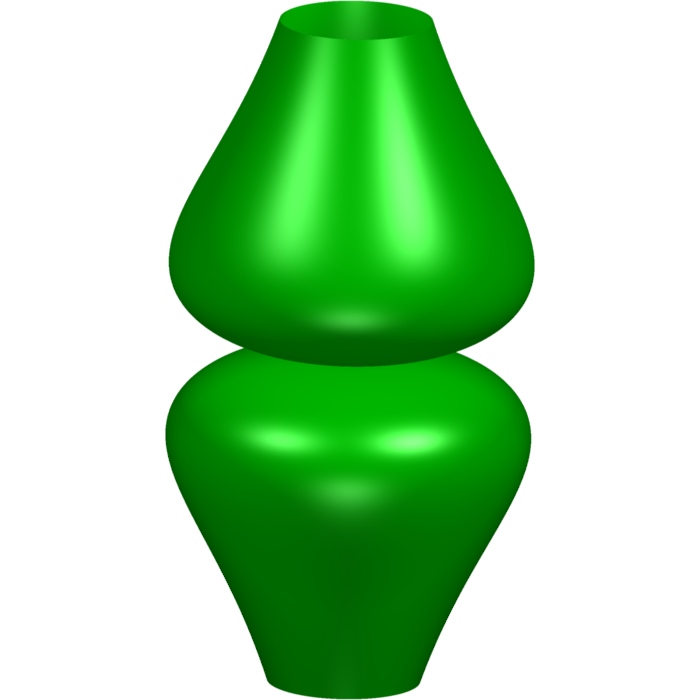}\hfill \includegraphics[width=3cm]{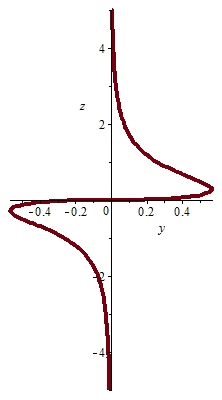}
\end{center}
\caption{Left: surface of revolution $\S$. Right: profile curve $\c^{\rm P}$ (Example \ref{ex-15})}
\end{figure}
$\c^{\rm P}\neq \c^{\rm M}$, and $\r^{\rm P}$ is  normal. Thus all the real part of $\S$, with the exception of the mirror curve, is covered by $\P(s,t)$.
Algorithm 1 returns $\r^{\rm M}$.
\end{example}

\para

\begin{example}\label{ex-19}
Let $\S$ be the surface of revolution generated by the profile curve
\[ \r^{\rm P}(t)=\left(0,{\frac {{t}^{3}}{{t}^{3}+1}},{\frac {{t}^{2}-1}{{t}^{2}+1}} \right). \]
\begin{figure}[h]
\begin{center}\includegraphics[width=6.2cm,height=6.2cm]{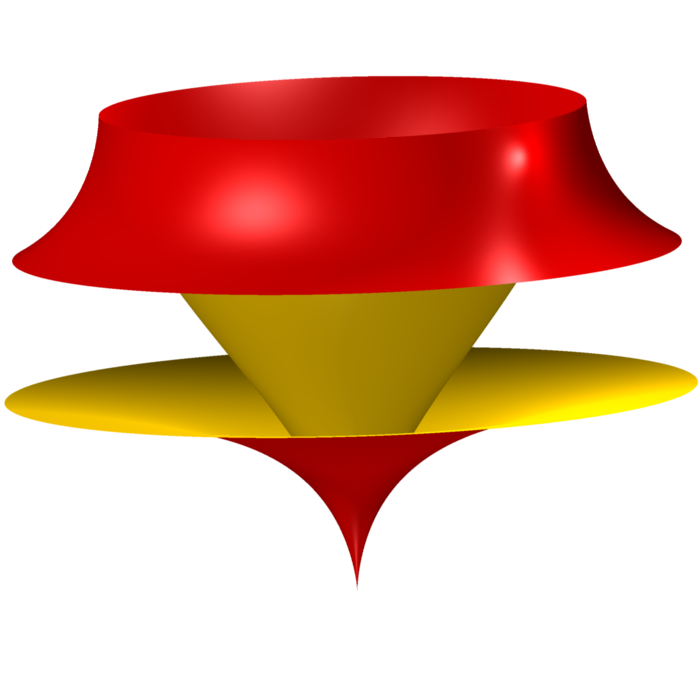}\hfill\includegraphics[width=6.3cm,height=5.6cm]{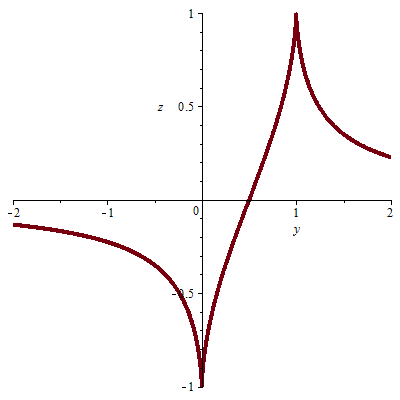}
\end{center}
\caption{Left: surface of revolution $\S$. Right: profile curve $\c^{\rm P}$ (Example \ref{ex-19})}
\end{figure}
$\c^{\rm P}\neq \c^{\rm M}$, and $\r^{\rm P}$ is not normal. Its critical point is $(0,1,1)$. Since $(0,-1,1)\not\in \c^{\rm P}$, Algorithm 1 returns $\r^{\rm M}\cup \cir(1,1)$.
\end{example}

\section{Extensions and future work}\label{sec-extensions}

In the previous sections, we have analyzed the missing area of the surface of revolution $\S$, when the parametrization $\P(s,t)$ takes values in the complex plane $\C^2$. Additionally, one could study the same problem but considering that the values of the parameters are taken only in the real plane $\R^2$. In \cite{Sendra2002a} this problem is analyzed for the case of plane curves, proving that, unless the curve has isolated singularities, one can always parametrize the curve surjectively just with values over $\R$. As future work one may extend these results on curves to the case of surfaces of revolution. Observe that, if the profile curve $\c^{\rm P}$ has an isolated singularity, a real circle on $S$ is generated, and it is only reachable by $\P$ with complex parameter values.

In addition, one may study the missing sets of other surface constructions in CAD, for instance swung surfaces (see e.g. \cite{qin}, \cite{zhao}); note that ruled surfaces are approached in \cite{SendraSevillaVillarino2014C}.

\section{Conclusions}

In this paper we prove that, for surfaces of revolution, the critical set can be taken, in the worst case, as the union of a curve and a circle; in Table \ref{tabla} we describe the possible critical sets. This set is easily and directly computable from the input profile curve parametrization. Moreover, the computations needed to deduce and describe the critical set are all done over the ground field and only require polynomial gcds, hence all is handled by means of linear algebra techniques. If $n$ is the degree of the profile parametrization $\r^{\rm P}$, the methods require $n$ gcds of degree $n$ univariate polynomials, and therefore the complexity of the method is dominated by $n^3$. Note that the complexity of the direct method is exponential.

\para

\begin{center}
\begin{tabular}{|c|c|c|c|}
  \hline
  & & &\\
  {\bf Is $\r^{\rm P}$ symmetric?} & {\bf Is $\r^{\rm P}$ normal?} & {\bf Real-critical set} & {\bf Example} \\
  \hline
  \hline
  Yes & Yes & Empty set & \ref{ex-11}\\
  \hline
  Yes & No & The critical point & \ref{ex-13}\\
  \hline
  No & Yes & The mirror curve & \ref{ex-15}\\
  \hline
  No & No & $\begin{array}{c} \mbox{The mirror curve union} \\
  \mbox{the cross section circle} \\  \mbox{at the critical point} \end{array}$ & \ref{ex-19}\\
  \hline
\end{tabular}
\captionof{table}{Possible real-critical sets}\label{tabla}
\end{center}

\section*{Acknowledgments}

This work was developed, and partially supported, by the Spanish \emph{Ministerio de Eco\-no\-m{\'\i}a y Competitividad} under Project MTM2011-25816-C02-01; as well as Junta de Extremadura and FEDER funds (group FQM024). The first and third authors are members of the
Research Group ASYNACS (Ref. CCEE2011/R34). The second author is a member of the research group GADAC (U. of Extremadura).

\section*{Appendix: proofs of the main results}

Let $f(y,z)$ be the defining polynomial of the profile curve $\c^{\rm P}$. In \cite{SanSegundo-Sendra} a description of the defining polynomial of $\S$ is given in terms of $f$. Collecting terms of odd and even degree in $y$ we can write $f(y,z)=A(y^2,z)+yB(y^2,z)$.
The implicit equation of $\S$ depends on whether $B$ is zero or not (i.e.  on whether $\c^{\rm P}$ is symmetric or not):
\begin{itemize}
 \item[] {\sf [Symmetric case]} $F(x,y,z)=A(x^2+y^2,z)$ when $B=0$.
 \item[] {\sf [Non-Symmetric case]} $F(x,y,z)=A^2(x^2+y^2,z)-(x^2+y^2)B^2(x^2+y^2,z)$ otherwise.
\end{itemize}

In the next lemma we study the level curves of $\S$.
\begin{lemma}\label{lemma-level}
 The intersection of $\S$ with the plane $z=c$ is either empty, or a finite union of circles $\cir(\alpha,c)$ with $\alpha\neq 0$, or the pair of lines $\{x\pm\ii y=0,\ z=c\}$. Moreover, if $(x_{0},y_{0},z_{0})\in \S$ where $x_{0}^2+y_{0}^2=\alpha^2\neq 0$, then $\cir(\alpha,z_{0})\subset \S$.
\end{lemma}
\begin{proof}
Let us reason in the non-symmetric case; the symmetric case is similar. Let
$g(\lambda)=A^2(\lambda,c)-\lambda B^2(\lambda,c),$
where $\lambda=x^2+y^2$.  Then
\begin{enumerate}
\item if $g(\lambda)\in \C\setminus\{0\}$ then $\S\cap \{z=c\}=\{\emptyset\}$; observe that $g(\lambda)$ cannot be identically zero, because we have assumed that $\S$ is not a plane;
\item if $g(\lambda)\in \C[\lambda]\setminus\C$ and $g(0)=0$ then  $\{x\pm iy=0, z=c\}\subset \S$;
\item if $g(\lambda)\in \C[\lambda]\setminus\C$ and $g(\alpha^2)=0$ with $\alpha \neq 0$ then $\cir(\alpha, c)\subset \S$.
\end{enumerate}
Let $(x_{0},y_{0},z_{0})\in \S$ where $x_{0}^2+y_{0}^2=\alpha^2\neq 0$.  $\cir(\alpha,z_{0})$ can be parametrized as
$\rho(s):=(\rho_1(s),\rho_2(s),z_{0})$ with $\rho_1=2s \alpha/(s^2+1), \rho_2=\alpha(s^2-1)/(s^2+1)$.
Taking into account that $\rho_1(s)^2 +\rho_2(s)^2=\alpha^2$ one has
$F(\rho(s))=A^2(\alpha^2, z_{0})-\alpha^2B^2(\alpha^2, z_{0})=F(x_{0},y_{0},z_{0})=0$, so $\cir(\alpha, z_{0})\subset \S$.
\end{proof}

When a point $P\in\c^{\rm P}$ rotates around the $z$-axis it generates a circle in $\S$ except when $P$ belongs to the axis. In the following lemmas we analyze these cases.

\para

\begin{lemma}\label{lemma-pto-1}
 Let $P=\r^{\rm P}(t_0)$ with $p(t_{0})\neq 0$. Then $\cir(p(t_0), q(t_0))$ is  reachable by $\P$, except possibly the symmetric point $P_{*}=\r^{\rm M}(t_0)\in \c^{\rm M}$.
\end{lemma}
\begin{proof} Since $p(t_0)\neq 0$, then $\cir(p(t_{0}), q(t_{0}))$  can be parametrized as $\P(s,t_{0})$,
which covers the whole circle
with the exception of $P_{*}=\r^{\rm M}(t_0)$ (see Theorem 2 in \cite{Sendra2002a}).
\end{proof}

\para

\begin{lemma}\label{lema rectas no alcanzables}
$(0,0,z_0)\in \c^{\rm P}$ if and only if
 $\{x\pm \ii y=0, z=z_0\}\subset \S$. Moreover,
 \begin{enumerate}
 \item[(i)] If $(0,0,z_0)\in \c^{\rm P}$, the  lines $\{x\pm \ii y=0, z=z_0\}$ are not reachable by $\P$, except possibly the point $(0,0,z_0)$.
 \item[(ii)] If $(x_0,y_0,z_0)\in\S$ with $x_0^2+y_0^2=0$, then $(0,0,z_0)\in\c^{\rm P}$.
 \item[(iii)] If $(x_0,y_0,z_0)\in\S$ with $x_0^2+y_0^2=0$, then $\{x\pm \ii y=0, z=z_0\}\subset \S$.
 \end{enumerate}
\end{lemma}
\begin{proof}
\noindent $1 \Rightarrow 2 )$ The lines $\{x\pm iy=0, z=z_0\}$ are parametrized by $\ell(\lambda)=(\mp \lambda \ii, \lambda, z_{0})$. Let us see that $\ell(\lambda)\subset \S$. Since $(0,0,z_0)\in \c^{\rm P}$ then $f(0,z_{0})=A(0,z_{0})=0$ (in the symmetric case) or $f(0,z_{0})=A^2(0,z_{0})=0$ (in the non-symmetric case). Moreover, since $(\mp \lambda \ii)^2+\lambda^2=0$, then $F(\ell(\lambda))=A(0,z_{0})=0$ and therefore $\ell(\lambda)\subset \S$.

\noindent $2\Rightarrow 1)$ Since $(0,0,z_{0})\in \S$ then $F(0,0,z_{0})=0$, and so $A(0,z_0)=0$. Thus, $f(0,z_0)=A(0,z_0)=0$ what implies that $(0,0,z_0)\in \c^{\rm P}$.

\noindent (i)  follows from the fact that the equality  $\P(s,t)=\ell(\lambda)$ only holds  for $\lambda =0$ and when $(0,0,z_0)$ is reachable by $r(t)$.

\noindent The proof of (ii) is analogous to the proof of the implication (2) $\Rightarrow$ (1).

\noindent (iii) follows from (ii) and the above equivalence.
\end{proof}

\begin{lemma}\label{lema puntos en c}
 Let $(x_0,y_0,z_0)\in \S$. Then $P^+=(0,\sqrt{x_{0}^{2}+y_{0}^{2}},z_0)\in \c^{\rm P}$ or $P^-=(0,-\sqrt{x_{0}^{2}+y_{0}^{2}},z_0)\in \c^{\rm P}$.
\end{lemma}
\begin{proof} Let $\alpha=\sqrt{x_{0}^{2}+y_{0}^{2}}$. If the implicit equation of $\S$ is $F(x,y,z)=A(x^2+y^2,z)$, then $F(x_0,y_0,z_0)=A(\alpha^2,z_0)=f(\pm \alpha, z_0)=0,$
and  $P^{\pm}\in \c^{\rm P}$. If the implicit equation of $\S$ is $F(x,y,z)=A^2(x^2+y^2,z)-(x^2+y^2)B^2(x^2+y^2,z)$, then
$F(x_0,y_0,z_0)=A^2(\alpha^2,z_0)-\alpha^2B^2(\alpha^2,z_0)=(A(\alpha^2,z_0)-\alpha B(\alpha^2,z_0))(A(\alpha^2,z_0)+\alpha B(\alpha^2,z_0))=0.$
When the first factor vanishes  then $P^-\in\c^{\rm P}$, and if the second vanishes then $P^+\in\c^{\rm P}$.
\end{proof}

\begin{lemma}\label{lema circulo alcanzable}
 Let $\cir(\alpha,c)\subset \S$, with $\alpha\neq 0$, and let $P_1=(0,\alpha,c)$, $P_2=(0,-\alpha,c)$. The following statements are equivalent:
\begin{enumerate}
\item  $\cir(\alpha,c)$ contains at least one point reachable by $\P$.
\item  $\cir(\alpha,c)$ is reachable by $\P$ except, at most, one of the points $P_i$.
\item One of the points $P_i$ is reachable by $\r^{\rm P}(t)$.
\end{enumerate}
\end{lemma}
\begin{proof}
\noindent $1\Rightarrow 2)$ Let $\P(s_0,t_0)=(x_0,y_0,c)\in \cir(\alpha,c)$.
Since $x_0^2+y_0^2=p^2(t_0)=\alpha^2\neq 0$, then
$\P(s,t_0)$
parametrizes $\cir(\alpha,c)$. Therefore it is reachable, except for the point $(0,-p(t_0),c)$. If $p(t_0)=\alpha$ we miss at most $P_2$, and if $p(t_0)=-\alpha$ at most the point $P_1$ is missed.

\noindent $2\Rightarrow 3)$ We assume w.l.o.g. that $P_1=\P(s_0,t_0)=(0,\alpha,c)$.
Taking into account that
\[
\left(\frac{2s_0}{1+s_0^2}p(t_{0})\right)^2+\left(\frac{1-s_0^2}{1+s_0^2}p(t_{0})\right)^2=p^2(t_0)=\alpha^2\neq 0\,\,\,\mbox{and}\,\,\, \frac{2s_0}{1+s_0^2}p(t_{0})=0
\]
one has that $s_0=0$, $p(t_0)=\alpha$, and hence $\r^{\rm ^P}(t_0)=P_1$.

\noindent $3\Rightarrow 1)$ We assume w.l.o.g. that $P_1=\r^{\rm P}(t_0)$, that is, $(0,\alpha,c)=(0,p(t_0),q(t_0))$. Then $P_1=\P(0,t_0)\in\cir(\alpha,c)$.
\end{proof}
\para

In this situation, let $\crit$ denote a critical set of $\P$, and let $N=(x_0,y_0,z_0)\in \S$ be non-reachable. If $x_{0}^{2}+y_{0}^{2}=0$, by Lemma \ref{lema rectas no alcanzables}, $N\in \{x\pm \ii y =0, z=z_0\}\subset \crit$ and $(0,0,z_0)\in \c^{\rm P}$.
 Let $x_{0}^{2}+y_{0}^{2}=\alpha^2\neq 0$. By Lemma \ref{lemma-level}, $\cir(\alpha,z_0)\subset \S$ and, by Lemma \ref{lema puntos en c}, $P^+=(0,\alpha,z_0)\in\c^{\rm P}$ or $P^-=(0,-\alpha,z_0)\in \c^{\rm P}$. We distinguish two cases:

\noindent (1) Assume $P^{\pm}\in \c^{\rm P}$. Note that, by \cite[Theorem 2]{Sendra2002a}, al least one of them is reachable by $\r^{\rm P}(t)$, and hence by $\P(s,t)$. Assume  w.l.o.g. that $P^+$ is reachable by $\r^{\rm P}$. By Lemma \ref{lema circulo alcanzable}, $\cir(\alpha,z_0)$ is reachable by $\P$ with the possible exception of $P^-$. Since $N\in \cir(\alpha,z_0)$ and it is non-reachable, then $P^-=N$. Moreover, $P_*=P^-$ (see Lemma \ref{lemma-pto-1} for the definition of $P_*$). Thus, $N\in \c^{\rm M}\subset \crit$.

\noindent (2) Assume either $P^+$ or $P^-$ belong to $\c^{\rm P}$. Say  w.l.o.g. that $P^+\in \c^{\rm P}$:
 if $P^+=\r^{\rm P}(t_0)$, by Lemma \ref{lemma-pto-1}, $\cir(\alpha,z_0)$ is reachable except at $P^-=N$. So, $N\in \c^{\rm M}\subset \crit$. On the other hand, if $P^+$ is not reachable by $\r^{\rm P}(t)$, by Lemma \ref{lema circulo alcanzable}, $N\in \cir(\alpha,z_0)\subset \crit$.

As a consequence of this analysis we have proven Theorems \ref{theorem-sym-real}, \ref{theorem-non-sym-real}, \ref{theorem-complex}.


\begin{thebibliography}{11}

\bibitem{Agoston} Agoston M.K. (2005).
\newblock Computer graphics and geometric modeling--implementation and algorithms.
\newblock Springer, Berlin.

\bibitem{juange}
Alc\'azar J.G.,  Hermoso C.,  Muntingh G. (2014),
\newblock  Detecting Symmetries of Rational Plane and Space Curves,
\newblock {\em Computer Aided Geometric Design}, to appear.

\bibitem{andradas} Andradas C., Recio T. (2007),
\newblock Plotting missing points and branches of real parametric
curves.
\newblock {\em Applicable Algebra in Engineering, Communication and Computing},
18(1--2), 107--126.

\bibitem{ARSTV} Andradas C., Recio T., Sendra J.R., Tabera L.F., Villarino C. (2014).
\newblock Reparametrizing Swung Surfaces over the Reals.
\newblock {\em Applicable Algebra in Engineering, Communication and Computing} 25:39--65. 

\bibitem{bajaj} 
 Bajaj C.L.,  Royappa A.V. (1995).
 \newblock Finite representations
of real parametric curves and surfaces.
\newblock {\em Internat. J.
Comput. Geom. Appl.,} 5(3):313--326.

\bibitem{Gao}
Chou S. C., Gao X. S. (1991).
\newblock On the normal parametrization of curves and surfaces.
\newblock Int. J. Comput. Geom. Appl., 1, 125--136.


\bibitem{qin}
Qin H., Terzopoulos D. (1995),
\newblock Dynamic NURBS swung surfaces for physics-based shape design.
\newblock {\rm Computer Aided Design} 27(2), 111--127.

\bibitem{marsh}
Marsh D. (2005),
\newblock Applied Geometry for Computer Graphics and CAD (second edition)
\newblock Springer.


\bibitem{SanSegundo-Sendra}
F. San Segundo, J. R. Sendra.
\newblock Offsetting Revolution Surfaces.
\newblock {\em ADG 2008, LNCS 6301 (T. Sturm and C. Zengler Eds.)}, pp. 179--188. Springer-Verlag Berlin Heidelberg (2011) ISBN: 978-3-642-21045-7.

\bibitem{Sendra2002a}
J.~R. Sendra.
\newblock Normal parametrizations of algebraic plane curves.
\newblock {\em J. Symbolic Comput.}, 33(6):863--885, 2002.

\bibitem{SendraSevillaVillarino2014a}
J.~R. Sendra, D.~Sevilla, and C.~Villarino.
\newblock Covering of surfaces parametrized without projective base points.
\newblock Proceeding ISSAC '14, pages 375-380
ACM Press, ISBN: 978-1-4503-2501-1 doi: 10.1145/2608628.2608635


\bibitem{SendraSevillaVillarino2014b}
J.~R. Sendra, D.~Sevilla, and C.~Villarino (2014).
\newblock Some results on the surjectivity of surface parametrizations.
\newblock Proc. RICAM Workshop, Springer LNCS, J.Schicho, M.Weimann, J.Gutierrez (eds) (to appear)

\bibitem{SendraSevillaVillarino2014C}
J.~R. Sendra, D.~Sevilla, and C.~Villarino (2014).
\newblock Covering Rational Ruled Surfaces.
\newblock arXiv:1406.2140v2.


\bibitem{SWP} J.R. Sendra, F. Winkler, S. P\'erez-D\'{\i}az.
\newblock {Rational Algebraic Curves: A Computer Algebra Approach}.
\newblock {\em Springer-Verlag Heidelberg. In series Algorithms and Computation  in Mathematics.}  Volume 22. 2007

\bibitem{Shragai} N. Shragai, G. Elber.
\newblock Geometric covering.
\newblock Computer-Aided Design 45 (2013) 243--251.

\bibitem{zhao}
Y. Zhao, Y. Zhou, J. L. Lowther, C.-K. Shene (1999).
\newblock Cross-sectional design: a tool for computer graphics and computer-aided design courses.
\newblock Frontiers in Education Conference, 1999. FIE '99. 29th Annual pp. 12B3/1 - 12B3/6 vol. 2. IEEE. DOI 10.1109/FIE.1999.841576

\end{thebibliography}
\end{document}